\documentclass[11pt]{amsart}

\usepackage{amsxtra,amsaddr,amsfonts}
\usepackage[latin1]{inputenc}
\usepackage{graphicx,color}
\usepackage{amsmath,amssymb,latexsym,amsthm,mathrsfs}
\usepackage{algorithm,algpseudocode}

\newtheorem{theorem}{Theorem}[section]
\newtheorem{thm}[theorem]{Theorem}
\newtheorem{lemma}[theorem]{Lemma}

\makeatletter
\let\c@equation\c@theorem
\makeatother
\numberwithin{equation}{section}

\newcommand{\cA}{\mathcal{A}} \newcommand{\cB}{\mathcal{B}} \newcommand{\cC}{\mathcal{C}}      \newcommand{\cI}{\mathcal{I}}       
 \newcommand{\cR}{\mathcal{R}}        

\renewcommand{\dfrac}[2]{\lower0.15ex\hbox{\large$\frac{#1}{#2}$}}
\newcommand{\V}{V\mkern-3mu}
\newcommand{\card}[1]{\mathopen| #1\mathclose|}

\title{\Large $R(5,5) \leq 48$}
\author{Vigleik Angeltveit${}^1$ and Brendan D. McKay${}^2$}
\address{${}^1$Mathematical Sciences Institute; 
${}^2$Research School of Computer Science\\
Australian National University, Canberra, ACT 2601, Australia}
\email{vigleik.angeltveit@anu.edu.au, brendan.mckay@anu.edu.au}
\begin{document}
 
\begin{abstract}
We improve the upper bound on the Ramsey number $R(5,5)$ from $R(5,5) \leq 49$ to $R(5,5) \leq 48$.
We also complete the catalogue of extremal graphs for $R(4,5)$.
\end{abstract}

\maketitle

\section{Introduction}
The Ramsey number $R(s,t)$ is defined to be the smallest $n$ such that every graph of order $n$ contains either a clique of $s$ vertices or an independent
set of $t$ vertices.

\begin{thm} \label{t:main}
The Ramsey number $R(5,5)$ is less than or equal to $48$.
\end{thm}

The history of $R(5,5)$ is provided in~\cite{McRa97}.  The lower bound
of 43 established constructively by Exoo~\cite{Ex89} is still the best.
The previous best upper bound of 49 was proved by McKay and Radziszowski~\cite{McRa97}.
By Theorem~\ref{t:main} we now have $43 \leq R(5,5) \leq 48$.

The actual value of $R(5,5)$ is widely believed to be $43$, because a lot of computer resources have been expended in an unsuccessful attempt to construct a Ramsey(5,5)-graph of order 43~\cite{McRa97}.
As additional evidence, we can report that, in unpublished 2014 work,
Lieby and the second author proved
that any Ramsey(5,5) graph on 42 vertices other than the 656 reported
in~\cite{McRa97} do not share a 37-vertex subgraph with any of the~656.

The proof of Theorem \ref{t:main} is via computer verification, checking approximately two trillion separate cases.
We wrote two independent programs to carry out the calculation, to minimise the chance of any computer bugs affecting our results.

\section{Outline of the proof of Theorem \ref{t:main}}
Let $\cR(s,t,n)$ denote the set of isomorphism classes of graphs of order $n$ without an $s$-clique or independent $t$-set,
and $\cR(s,t)=\bigcup_n \cR(s,t,n)$.
The main idea is that given a graph $F \in \cR(5,5,48)$, a large subgraph of it must be obtained by gluing together two graphs in $\cR(4,5,24)$ along a graph in $\cR(3,5,d)$ for some $d$.

A list of 350,904 graphs in $\cR(4,5,24)$ was compiled by McKay and Radziszowski \cite{McRa95} in 1995, and our first task was to complete their list. This was actually the most time-consuming part of the project.

\begin{thm} \label{t:R45}
$\card{\cR(4,5,24)} = 352{,}366$.
\end{thm}

\begin{figure}[ht]
\centering
\includegraphics[scale=0.7]{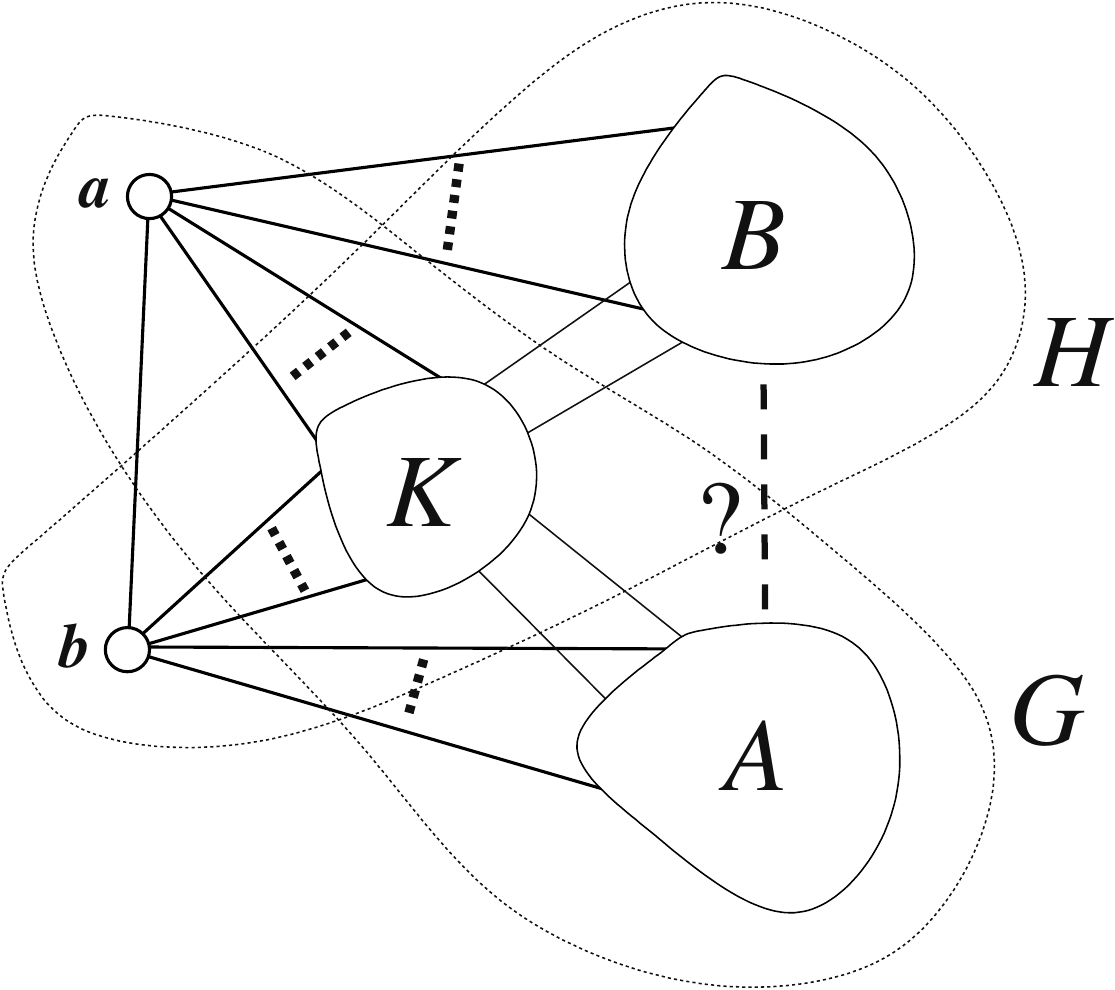}
\caption{$G$ and $H$ are given graphs from $\cR(4,5,24)$ that
  overlap in a graph $K$.  The problem is to choose the edges
  between $A$ and $B$ so that the whole is in $\cR(5,5)$.\label{r55pic}}
\end{figure}

We now explain the main proof idea in more detail.
For a graph $F$, $\V F$ is the vertex-set of $F$, $N_F(w)$ is the 
neighbourhood of vertex $w\in \V F$ and $F[W]$ is the subgraph of $F$ induced
by~$W\subseteq \V F$.
First, note that because $R(4,5)=25$ \cite{McRa95}, every vertex in a graph $F \in
\cR(5,5,48)$ must have degree $23$ or $24$. By replacing $F$ by its complement if
necessary we can assume that $F$ has at least $24$ vertices of degree $24$.
Hence $F$ must have two adjacent vertices $a,b$ of degree $24$. Define
\begin{align*}
 G &= F[N_F(b)], \\
 H &= F[N_F(a)], \\
 K &= F[\V G \cap \V H].
\end{align*}
In words, $G$ is the subgraph of $F$ induced by the $24$ vertices adjacent to~$b$
(this includes~$a$ but not~$b$), $H$ is the subgraph induced by the vertices
adjacent to~$a$, and $K$ is the intersection of $G$ and~$H$.
Please see Figure~\ref{r55pic}.
Note that $G, H \in \cR(4,5,24)$ and that $K \in \cR(3,5,d)$ for some $d$. Because
$R(3,5) = 14$ we must have $d \leq 13$, and $d$ is also equal to the degree of $a$ in $G$ and the degree of $b$ in $H$.

To reconstruct $F[\V G \cup \V H]$, which is a graph with $48-d$ vertices, from $G$, $H$ and $K$, it suffices to specify how $K$ is a subgraph of $G$ and $H$, and whether or not we have an edge between $x$ and $y$ for $x \in \V G-\V K-\{a\}$ and $y \in \V H-\V K-\{b\}$; i.e. between parts labelled $A$ and $B$ in Figure \ref{r55pic}.
We call this procedure \textit{gluing}.
For each inclusion of $K$ into $G$ and $H$ there are $2^{(23-d)^2}$ ways of gluing $G$ and $H$ along $K$, but we will only consider gluings that could give a graph in $\cR(5,5,48-d)$.

For $K \in \cR(3,5,d)$, define
\[
 \cR(4,5,24,K) = \{ (G,a) \mid G \in \cR(4,5,24), a \in \V G, G[N_G(a)] \cong K\}.
\]
We will call $(G,a)$ a \textit{pointed graph of type $K$}. Our proof of Theorem \ref{t:main} consists of the following steps.
\newline

\noindent
\textbf{Step 1:}
We completed the list of graphs in $\cR(4,5,24)$ compiled by McKay and Radziszowski,  thereby proving Theorem \ref{t:R45}. This was done by a straightforward (but computationally expensive) extension of the method in~\cite{McRa95}. While that calculation would have taken too long in 1995,
it was doable in 2016.
\newline

\noindent
\textbf{Step 2:}
For each $K \in \cR(3,5,d)$ with $d \leq 11$ and for each pair
$(G,a), (H,b)\in\cR(4,5,24,K)$, we used a computer program to calculate all ways of gluing $G$ and $H$ along $K$. Note that this consisted of one gluing problem for each automorphism of $K$.
\newline

\noindent
\textbf{Step 3:}
For each graph generated in Step 2, we used another program which attempts
in all possible ways to add one vertex while staying within $\cR(5,5)$.
Since this was never possible, none of the graphs generated in Step~2 are subgraphs of a graph in $\cR(5,5,48)$.

\begin{lemma}\label{l:Step123ok}
 Execution of Steps 1--3 is sufficient to prove Theorem~\ref{t:main}.
\end{lemma}
\begin{proof}
 Suppose $F\in\cR(5,5,48)$.  We first prove that either $F$ or its complement
 has a vertex of degree~24 adjacent to at least 12 other vertices of degree~24.
 Suppose that $F$ is a counterexample to this claim, and let $W\subseteq \V F$ be
 its vertices of degree~24.  Since $F[W]$ has maximum degree 11, there are
 at least $e_1=13\,\card W$ edges between $W$ and $\V F\setminus W$ in~$F$.
 Similarly, there are at least $e_2=13\,(48-\card{W})$ edges between
 $W$ and $\V F\setminus W$ in~$\bar F$.
 However, this is impossible since $e_1+e_2=13\times 48=624$
 and $\card W(48-\card W)\le 24^2=576$.
 
 So let $b$ be a vertex of $F$ of degree 24 that is adjacent to at least 12
 other vertices of degree~24 and define $G=F[N_F(b)]$.  From the
 $\cR(4,5,24)$ catalogue we find that $G$ has at most 8 vertices of
 degree more than~11, so we can choose $a\in N_F(b)$ that has
 degree 24 in $F$ and degree at most 11 in~$G$.  Define
 $H=F[N_F(a)]$.  Then the gluing of $(G,a)$ and $(H,b)$ in
 Step~2 will find a subgraph of~$F$ and the failure of one point
 extension in Step~3 will show that $F$ doesn't exist.
\end{proof}

\begin{table}[ht]
\centering
\begin{tabular}{c|ccccc|r}
$e$&$i_3$&$i_4$&$c_3$&$\delta$&$\varDelta$&count\\
\hline
116&356--368&225--262&123--128&8--9 &10--11&9 \\
117&346--362&216--253&122--132&8--9 &10--11&90 \\
118&340--360&206--251&120--136&6--9 &10--12&806 \\
119&332--356&198--247&124--140&6--9 &10--13&4358 \\
120&324--352&186--243&127--144&6--10&10--13&16346 \\
121&319--344&181--232&130--146&6--10&11--13&43457 \\
122&314--337&178--223&133--149&6--10&11--13&79678 \\
123&310--330&171--215&136--152&6--10&11--13&92504 \\
124&304--324&163--208&140--154&6--10&11--13&67209 \\
125&302--318&161--201&144--157&6--10&11--13&31996 \\
126&296--312&155--195&147--160&7--10&11--12&11485 \\
127&291--301&152--177&152--162&8--10&11--12&3401 \\
128&286--296&149--171&156--164&8--10&11--12&843 \\
129&281--290&146--165&162--166&9--10&11--12&147 \\
130&276--282&143--155&166--169&9--10&11--12&32 \\
131&270--270&143--149&172--172&10--10&11--11&3 \\
132&264--264&138--144&176--176&11--11&11--11&2 \\
\hline
all&264--368&138--262&120--176&6--11&10--13&352366
\end{tabular}
\medskip
\caption{Statistics for all $(4,5,24)$-graphs\label{R45table}}
\end{table}

\section{Step 1: Completing the list of graphs in $\cR(4,5,24)$}
McKay and Radziszowski~\cite{McRa95} produced a list of 350,904 such graphs, and proved that the list contains all graphs in $\cR(4,5,24)$ with
minimum degree is 6, 7 or 8, or maximum degree 12 or 13, or
if the graph is regular of degree~11.

To complete the catalogue it suffices to find those graphs with minimum degree~9
or~10.  We did this using the well-tested code from~\cite{McRa95} to glue together
graphs of type $\cR(3,5,9)$ and $\cR(4,4,14)$, and of types $\cR(3,5,10)$ and
$\cR(4,4,13)$.
Although this requires a very large number of graph pairs to be glued, it
is feasible when the graphs of type  $\cR(3,5,9)$ and $\cR(3,5,10)$ are
arranged in a tree structure that exhibits common subgraphs and
symmetries.  See~\cite{McRa95} for details. All graphs in $\cR(4,5,24)$
with a vertex of
degree~9 or~10 were found, to increase the overlap with~\cite{McRa95}
for checking purposes.  This took about 1.5 core-years of computer
time and discovered 1462 new graphs in $\cR(4,5,24)$; recall that the
search in~\cite{McRa95} was not intended to be complete.

Then we devoted another 6 core-months to sanity-checking
of the completed catalogue.  As an example, let $\cA'$ be the set of all
neighbourhoods of a vertex of degree 9 or 10 in the 1462 new graphs,
and let $\cB'$ be the set of all complementary
neighbourhoods of the same vertices in those graphs.
Then, using a completely separate program, we constructed all graphs
in $\cR(4,5,24)$ with a vertex having a neighbourhood in $\cA'$ and a
complementary neighbourhood in~$\cB'$.  Only known graphs appeared.
We also proved, with a separate computation, that if there are any
graphs in $\cR(4,5,24)$ but not in the catalogue, they do not share any
21-vertex subgraph with a graph in the catalogue.

Summary statistics of the catalogue, to complete~\cite[Table~4]{McRa95}, are provided in Table~\ref{R45table}; $e$ is the number of edges,
$i_k$ is the number of independent sets of size~$k$, $c_3$ is the
number of triangles, and $\delta,\varDelta$ are the minimum and
maximum degrees.  The graphs themselves are available at~\cite{RamseyWeb}.

\section{The structure of $\cR(4,5,24,K)$}
The neighbourhood of a vertex $a$ of degree $d$
in a pointed graph
$(G,a)\in\cR(4,5,24,K)$ is the graph $K\in\cR(3,5,d)$.
However not all graphs in $\cR(3,5)$ appear in pointed graphs.
In Table~\ref{PointedTable}, we show the number of graphs~$K$
which occur at least once and the total number of pointed
graphs for each~$d$.
Note that we have not used the automorphism group of~$G$,
so some of the pointed graphs are isomorphic.
The great majority of graphs in $\cR(4,5,24)$ have trivial
automorphism group, so we gave up the small available
speedup (estimated at 3\%) in order to have fewer steps in
the computation.
The total of 8,456,784 in the table is $24\times\card{\cR(4,5,24)}$.

\begin{table}[ht]
\centering
\begin{tabular}{c|cc|r}
$d$ & $\card{\cR(3,5,d)}$ & occurring & count~~ \\
\hline
1--5 & 21 & 0 & 0 \\
6 & 32 & 2 & 1979 \\
7 & 71 & 11 & 7497 \\
8 & 179 & 88 & 64395 \\
9 & 290 & 240 & 832288 \\
10 & 313 & 294 & 4651124 \\
11 & 105 & 103 & 2800499 \\
12 & 12 & 11 & 97968 \\
13 & 1 & 1 & 1034 \\
\hline
all & 1029 & 750 & 8456784 
\end{tabular}
\medskip
\caption{Counts of pointed graphs\label{PointedTable}}
\end{table}
 
The number of pointed graphs in $\cR(4,5,24,K)$ for
$K\in\cR(3,5,{\le}11)$ varies greatly: from~0 to~526,073, the latter
from a rather irregular graph of order~11 and 21 edges.
For Step~2 we take two pointed graphs
$(G,a),\allowbreak(H,b)\in\cR(4,5,24,K)$ and overlap them so that their
common subgraph~$K$ coincides.  This can be done in one
distinct way for each automorphism of~$K$ (again ignoring some
small reductions arising from automorphisms of~$G$ and~$H$).
Most graphs $K$ have only trivial automorphisms but some
have  large automorphism groups, the largest having order 1152
(a vertex-transitive quartic graph of order~8).

Taking the wildly varying sizes of $\cR(4,5,24,K)$ as well as the automorphism groups of the various $K$ into account we needed to solve approximately 2 trillion gluing problems. While that is certainly a lot, we were able to perform hundreds of thousands of such gluings per second per core. The whole calculation took approximately six core-months for one implementation
and two core-months for the other.

\section{Step 2. Finding all ways to glue}
In order to ensure correctness, the list of pointed graphs was
prepared independently by the two authors and all the gluings were
performed by two programs written independently using different methods.
The decision to use two different methods rather than identifying the
fastest method and implementing it twice was based on the long-established
axiom of software engineering that different programmers tend to make
the same errors when faced with the same task.

Now we will describe the two different methods for gluing $(G,a), (H,b) 
\allowbreak \in \cR(4,5,24,K)$
after they are overlapped at the common subgraph~$K$.
Because of the large number of calculations needed, the naive approach of
deciding one unknown adjacency at a time takes far too long.

Define $d'=23-d$. 
Suppose $K$ has vertices $v_0,\ldots,v_{d-1}$, $G$ has vertices $v_0,\ldots,v_{d-1}, a, a_1,\ldots,a_{d'}$ and $H$ has vertices $v_0,\ldots,v_{d-1}, b, b_1,\ldots,b_{d'}$.
Note that the vertices $a$ and $b$ cannot participate in any $5$-cliques or independent $5$-sets by the construction.
To specify a gluing it suffices to specify whether or not $a_i$ and $b_j$ are connected by an edge for $1 \leq i, j, \leq d'$.
We will record this data in a $d' \times d'$ matrix
$M$ with entries $0$ (for no edge) and $1$ (for edge).

Define a \textit{potential $(r,s,t)$-clique} to be
$r$ vertices $w_1,\ldots,w_r$ in $\V K$, $s$ vertices $x_1,\ldots,x_s$
in $\V G-\V K-\{a\}$, and $t$ vertices $y_1,\ldots,y_t$ in $\V H-\V K-\{b\}$ such that
\[
 \{w_1,\ldots,w_r,x_1,\ldots,x_s\}
\]
is an $(r+s)$-clique in $G$ and
\[
 \{w_1,\ldots,w_r,y_1,\ldots,y_t\}
\]
is an $(r+t)$-clique in $H$. Define a \textit{potential
independent $(r,s,t)$-set} similarly.
The following lemma is immediate.

\begin{lemma}\label{l:newcliques}
 A $d' \times d'$ 0-1 matrix $M = (m_{ij})$ defines a gluing if and only if
 \begin{enumerate}
 \item For each potential $(r,s,t)$-clique with $r+s+t = 5$, $m_{x_iy_j}=0$
 for some $1\le i\le s,1\le j\le t$.   
 (This is needed for  $(1,2,2)$, $(0,2,3)$ and $(0,3,2)$.)
 \item For each potential independent $(r,s,t)$-set with $r+s+t = 5$,
  $m_{x_iy_j}=1$ for some $1\le i\le s,1\le j\le t$.
 (This is needed for $(3,1,1)$, $(2,1,2)$, $(2,2,1)$, $(1,1,3)$, $(1,2,2)$,
 $(1,3,1)$, $(0,2,3)$ and $(0,3,2)$.) 
\end{enumerate}
\end{lemma}
\begin{proof}
  Please refer to Figure~\ref{r55pic} and consider a set $W$ of size~5.
  For $W$ to be a clique in the completed graph, it must overlap both
  $K\cup A$ and $K\cup B$, and the pairs of vertices in each those
  intersections must be edges.  That implies it is one of the
  potential $(r,s,t)$-cliques
  listed in part~(1), and to prevent $W$ from being a clique in the
  completed graph we need to include a non-edge.  The case of
  an independent set is similar.
\end{proof}

\medskip
The two gluing methods are logically similar but implemented
very differently.
The first gluing method expands on the method in \cite{McRa95}. Define an \emph{interval} to be a set of the form $I = \{X \, | \, B \subseteq X \subseteq T\}$, where $B$ and $T$ are subsets of $\{a_1,\ldots,a_{d'}\} \times \{b_1,\ldots,b_{d'}\}$. We write $I=[B,T]$. We represent $I$ by two $d' \times d'$ matrices with coefficients in $\{0,1\}$.

Given an interval $[B,T]$, we define collapsing rules as follows. There are $11$ in total, one for each of the triples in Lemma \ref{l:newcliques} above. The special event FAIL means that there is no $X\in[B,T]$ which corresponds to a 
proper gluing.

\smallskip
\noindent \textbf{Rule} $K_{1,2,2}$. Suppose $\{w_1, x_1, x_2, y_1, y_2\}$ is a
potential $(1,2,2)$-clique.
\begin{algorithmic}
\If {$(x_1,y_1), (x_1,y_2), (x_2,y_1), (x_2,y_2) \in B$} FAIL
\ElsIf{ $(x_1,y_1), (x_1,y_2), (x_2,y_1) \in B$} $T := T-(x_2,y_2)$
\ElsIf{ $(x_1,y_1), (x_1,y_2), (x_2,y_2) \in B$} $T := T-(x_2,y_1)$
\ElsIf{ $(x_1,y_1), (x_2,y_1), (x_2,y_2) \in B$} $T := T-(x_1,y_2)$
\ElsIf{ $(x_1,y_2), (x_2,y_1), (x_2,y_2) \in B$} $T := T-(x_1,y_1)$
\EndIf
\end{algorithmic}

\smallskip\noindent
The collapsing rules for $K_{0,2,3}$ and $K_{0,3,2}$ are similar.
In each case, the rule says that if 5 vertices include 9 edges, then
the remaining vertex pair must not be an edge.

\smallskip\noindent
\textbf{Rule} $E_{3,1,1}$. Suppose $\{w_1, w_2, w_3, x_1, y_1\}$ is a
potential independent $(3,1,1)$-set.
\begin{algorithmic}
\If{ $(x_1,y_1) \not \in T$} FAIL
\Else{ $B := B \cup (x_1,y_1)$.}
\EndIf
\end{algorithmic}

\smallskip\noindent
The collapsing rules for the other potential
independent sets from Lemma~\ref{l:newcliques} are once again similar.

We start the search with a single interval $I=[B,T]$ with $B = \varnothing$ and $T = \{a_1,\ldots,a_{d'} \} \times \{b_1,\ldots,b_{d'} \}$, and we note that the collapsing rule $E_{3,1,1}$ can be applied even in this case. Each time we add an edge to $B$ or remove an edge from $T$ the number of possible gluings is cut in half.

After applying these collapsing rules repeatedly, we must eventually encounter either FAIL or a stable situation. The discussion in \cite{McRa95} applies, and the final state is independent of the order of the application of the collapsing rules.

If we do not encounter FAIL, we pick some $(a_i,b_j)$ with $(a_i,b_j) \not \in B$
and $(a_i,b_j) \in T$, and consider the cases $I=[B, T-(a_i,b_j)]$ and $I=[B \cup (a_i,b_j), T]$ separately.

\medskip
The second method applies an equivalent procedure using 
data structures familiar from the constraint satisfaction area.
Each entry $m_{ij}$ of $M$ is a \textit{variable}, with value FALSE,
TRUE or UNKNOWN, while each
set $\{x_1,\ldots,x_s\}\times\{y_1,\ldots,y_t\}$ is a \textit{clause}.
Clauses from potential $(r,s,t)$-cliques can't have all their variables TRUE,
while clauses from potential independent $(r,s,t)$-sets can't have all their
variables FALSE.
Each variable $\alpha$ has a list $\cC(\alpha)$ of the clique
clauses which contain $\alpha$, and a list $\cI(\alpha)$ of the
independent set clauses which contain $\alpha$.
There is also a stack $S$ which maintains a set of
distinct variables on a last-in first-out basis.
Informally, at each moment $S$ contains those variables which
have been assigned FALSE or TRUE, but their clause lists have
not yet been scanned.

Initially, variables are set to TRUE if required by independent
$(3,1,1)$-set clauses, and UNKNOWN otherwise.  The variables
equal to TRUE are put onto $S$.  Then we execute the following
until it terminates.

\vskip 0pt plus 50pt\penalty-300\vskip 0pt plus -50pt

\begin{algorithmic}
  \While{$S\ne\emptyset$}
     \State Pop the top variable $\alpha$ off $S$
     \If {$\alpha=\mathrm{FALSE}$}
         \For {each clause $C\in\cI(\alpha)$}
            \If {all variables in $C$ are FALSE}
                \State \textbf{exit} FAIL
            \ElsIf {all variables in $C$ are FALSE\\
                 \hspace{10em} except for $\beta=\mathrm{UNKNOWN}$}
                 \State Set $\beta := \mathrm{TRUE}$ and
                    and push $\beta$ onto $S$
            \EndIf
         \EndFor
     \Else
         \For {each clause $C\in\cC(\alpha)$}
            \If {all variables in $C$ are TRUE}
                \State \textbf{exit} FAIL
            \ElsIf {all variables in $C$ are TRUE\\
                 \hspace{10em} except for $\beta=\mathrm{UNKNOWN}$}
                 \State Set $\beta := \mathrm{FALSE}$ and
                    and push $\beta$ onto $S$
            \EndIf
         \EndFor
     \EndIf
  \EndWhile
\end{algorithmic}

\smallskip
For good efficiency it is essential that variables be assigned values as they
enter the stack and not when they leave it.  Also, a good optimization is for
clauses to remember how many UNKNOWN variables they have.
If the algorithm terminates with ``\textbf{exit} FAIL'', there is no solution.
Otherwise, all the variables with value FALSE or TRUE have those
values in all solutions.  If there is any variable with value UNKNOWN,
we can choose one such variable and try FALSE and TRUE separately
with $S$ initialised to that variable only.
And so on, recursively.

\medskip
Both methods were very fast for $d\ge 8$, often performing
100,000 gluings per second per core, primarily because failure
occurred early most of the time.

For $d\le 7$, the methods as described could take much longer
since extremely large search trees with many useless branches could
be generated.  For those values of $d$ we used additional techniques.

For the first method, two techniques were used.
First, for each pair $(a_i,b_j) \in T-B$ we applied the collapsing rules to both $[B, T-(a_i,b_j)]$ and $[B \cup (a_i,b_j), T]$. If for some pair $(a_i,b_j)$ we arrived at FAIL in both cases we then concluded that there were no gluings. If $[B,T-(a_i,b_j)]$  led to FAIL then we replaced $[B,T]$ by $[B \cup (a_i,b_j), T]$, and if $[B \cup (a_i,b_j), T]$ led to FAIL then we replaced $[B,T]$ by $[B,T-(a_i,b_j)]$. This is of course more expensive than the original algorithm at each node of the search tree, but we found that for $6 \leq d \leq 7$ it was worth it.

Second, we ordered the pairs $(a_i,b_j)$ according to how many independent sets of type $(2,2,1)$ and $(2,1,2)$ they were contained in and started the binary search with a pair $(a_i,b_j)$ which was maximal in this sense. The advantage is that when considering $[B,T-(a_i,b_j)]$ the collapsing rules $E_{2,2,1}$ and $E_{2,1,2}$, which require only a single edge to be missing from $T$ in order to modify $B$, come into play as much as possible.

For the second method, instead of choosing an arbitrary UNKNOWN variable to branch on, we used an UNKNOWN variable which occurred in the
greatest number of clique clauses with all TRUE variables except two UNKNOWN variables, or independent set clauses with all FALSE variables except two UNKNOWN variables.  This is a
heuristic for how beneficial it is to assign FALSE or TRUE to the variable.

In both cases, these enhancements made the cost per node of the search tree much greater but, due to the smaller number of pointed graphs for small~$d$, the computation finished quickly enough.

\begin{figure}[p]
\def\Maa{ \begin{array}{c|}
\hbox to 0pt{\hss\normalsize $a$\kern0.7em}0\end{array} }
\def\Mad{ \begin{array}{cccccccccccc|} 1&1&1&1&1&1&1&1&1&1&1&1 \end{array} }
\def\Mac{ \begin{array}{ccccccccccc|} 1&1&1&1&1&1&1&1&1&1&1 \end{array} }
\def\Mab{ \begin{array}{cccccccccccc} 0&0&0&0&0&0&0&0&0&0&0&0 \end{array} }
\def\Mae{ \begin{array}{c} 1\end{array} }

\def\Mda{ \begin{array}{c|}\strut
 1\\1\\1\\1\\1\\1\\1\\1\\1\\1\\1\\1  \end{array} }
\def\Mdd{ \begin{array}{cccccccccccc|}\strut
0&1&0&1&0&0&1&1&0&0&1&0 \\ 1&0&0&1&0&0&0&1&0&0&0&1 \\ 0&0&0&0&1&1&0&1&1&0&1&0 \\
1&1&0&0&0&0&0&0&0&1&1&0 \\ 0&0&1&0&0&0&0&1&0&1&1&1 \\ 0&0&1&0&0&0&1&0&1&0&1&0 \\
1&0&0&0&0&1&0&0&1&1&0&1 \\ 1&1&1&0&1&0&0&0&1&0&0&1 \\ 0&0&1&0&0&1&1&1&0&1&0&0 \\
0&0&0&1&1&0&1&0&1&0&0&1 \\ 1&0&1&1&1&1&0&0&0&0&0&0 \\ 0&1&0&0&1&0&1&1&0&1&0&0
\end{array} }
\def\Mdc{ \begin{array}{ccccccccccc}\strut
0&1&1&0&1&1&0&0&0&0&0 \\ 1&0&0&1&1&1&0&0&0&1&0 \\ 0&0&0&1&1&0&1&0&1&0&0 \\
1&0&1&0&0&0&1&1&1&0&1 \\ 0&0&1&1&0&1&0&1&0&0&1 \\ 1&0&1&0&1&0&1&0&0&1&1 \\
0&1&0&1&1&0&1&1&0&0&0 \\ 0&1&1&0&0&0&1&1&0&1&0 \\ 0&1&1&1&0&1&0&0&1&1&0 \\
1&0&1&0&1&0&0&1&1&1&0 \\ 0&1&0&0&0&1&0&1&1&1&1 \\ 1&1&0&0&0&1&1&0&1&0&1
\end{array} }
\def\Mdb{ \begin{array}{cccccccccccc|}\strut
1&1&1&1&1&1&1&0&1&0&0&0 \\ 1&0&1&1&1&0&0&1&0&1&1&1 \\ 1&1&0&1&1&1&0&0&1&1&0&1 \\
0&1&0&1&0&1&1&1&1&0&1&0 \\ 1&1&1&0&0&0&1&1&0&1&0&0 \\ 1&0&1&0&1&1&0&1&1&0&1&1 \\
1&0&0&1&1&0&1&1&1&1&0&0 \\ 1&0&1&0&0&1&1&1&1&0&0&1 \\ 1&1&1&0&0&1&1&0&0&1&0&1 \\
0&1&1&1&1&0&1&0&0&0&1&1 \\ 0&1&0&0&0&1&0&1&0&1&0&1 \\ 0&1&1&0&1&1&0&0&1&1&1&1
\end{array} }
\def\Mde{ \begin{array}{c}\strut 0\\0\\0\\0\\0\\0\\0\\0\\0\\0\\0\\0
\end{array} }

\def\Mca{ \begin{array}{c|}\strut
 1\\1\\1\\1\\1\\1\\1\\1\\1\\1\\1 \end{array} }
\def\Mcd{ \begin{array}{cccccccccccc|}\strut
0&1&0&1&0&1&0&0&0&1&0&1 \\ 1&0&0&0&0&0&1&1&1&0&1&1 \\ 1&0&0&1&1&1&0&1&1&1&0&0 \\
0&1&1&0&1&0&1&0&1&0&0&0 \\ 1&1&1&0&0&1&1&0&0&1&0&0 \\ 1&1&0&0&1&0&0&0&1&0&1&1 \\
0&0&1&1&0&1&1&1&0&0&0&1 \\ 0&0&0&1&1&0&1&1&0&1&1&0 \\ 0&0&1&1&0&0&0&0&1&1&1&1 \\
0&1&0&0&0&1&0&1&1&1&1&0 \\ 0&0&0&1&1&1&0&0&0&0&1&1
\end{array} }
\def\Mcc{ \begin{array}{ccccccccccc|}\strut
0&0&0&0&0&0&1&0&0&1&1 \\ 0&0&0&0&0&0&0&1&1&0&1 \\ 0&0&0&0&1&1&1&0&0&0&1 \\
0&0&0&0&1&1&0&1&0&0&1 \\ 0&0&1&1&0&0&0&0&1&1&0 \\ 0&0&1&1&0&0&0&0&1&1&0 \\
1&0&1&0&0&0&0&1&1&0&0 \\ 0&1&0&1&0&0&1&0&0&1&0 \\ 0&1&0&0&1&1&1&0&0&0&0 \\
1&0&0&0&1&1&0&1&0&0&0 \\ 1&1&1&1&0&0&0&0&0&0&0
\end{array} }
\def\Mcb{ \begin{array}{cccccccccccc|}\strut
1&0&0&0&1&0&1&0&1&0&1&1 \\ 0&1&0&0&0&1&0&1&0&1&0&1 \\ 0&1&1&1&0&0&1&1&0&0&0&0 \\
1&0&0&1&1&1&0&1&1&1&0&0 \\ 1&1&0&1&1&0&0&0&0&0&1&1 \\ 1&1&1&0&0&1&1&0&0&1&0&0 \\
0&0&0&1&0&0&1&1&1&1&1&0 \\ 0&0&1&0&0&1&1&1&1&0&0&1 \\ 0&1&0&0&1&1&0&0&1&1&1&0 \\
0&0&1&0&1&0&0&1&0&1&1&1 \\ 0&0&0&1&0&1&0&1&0&0&1&1
\end{array} }
\def\Mce{ \begin{array}{c}\strut 1\\1\\1\\1\\1\\1\\1\\1\\1\\1\\1 \end{array} }

\def\Mba{ \begin{array}{c|} \strut0\\0\\0\\0\\0\\0\\0\\0\\0\\0\\0\\0
\end{array} }
\def\Mbd{ \begin{array}{cccccccccccc|}\strut
1&1&1&0&1&1&1&1&1&0&0&0 \\ 1&0&1&1&1&0&0&0&1&1&1&1 \\ 1&1&0&0&1&1&0&1&1&1&0&1 \\
1&1&1&1&0&0&1&0&0&1&0&0 \\ 1&1&1&0&0&1&1&0&0&1&0&1 \\ 1&0&1&1&0&1&0&1&1&0&1&1 \\
1&0&0&1&1&0&1&1&1&1&0&0 \\ 0&1&0&1&1&1&1&1&0&0&1&0 \\ 1&0&1&1&0&1&1&1&0&0&0&1 \\
0&1&1&0&1&0&1&0&1&0&1&1 \\ 0&1&0&1&0&1&0&0&0&1&0&1 \\ 0&1&1&0&0&1&0&1&1&1&1&1 
\end{array} }
\def\Mbc{ \begin{array}{ccccccccccc|}\strut
1&0&0&1&1&1&0&0&0&0&0 \\ 0&1&1&0&1&1&0&0&1&0&0 \\ 0&0&1&0&0&1&0&1&0&1&0 \\
0&0&1&1&1&0&1&0&0&0&1 \\ 1&0&0&1&1&0&0&0&1&1&0 \\ 0&1&0&1&0&1&0&1&1&0&1 \\
1&0&1&0&0&1&1&1&0&0&0 \\ 0&1&1&1&0&0&1&1&0&1&1 \\ 1&0&0&1&0&0&1&1&1&0&0 \\
0&1&0&1&0&1&1&0&1&1&0 \\ 1&0&0&0&1&0&1&0&1&1&1 \\ 1&1&0&0&1&0&0&1&0&1&1
\end{array} }
\def\Mbb{ \begin{array}{cccccccccccc}\strut
0&1&0&0&0&0&1&1&1&0&1&0 \\ 1&0&0&0&0&0&0&1&1&0&0&1 \\ 0&0&0&1&1&1&0&0&1&0&1&0 \\
0&0&1&0&0&0&0&0&1&1&1&1 \\ 0&0&1&0&0&1&1&0&1&1&0&0 \\ 0&0&1&0&1&0&1&0&0&0&1&0 \\
1&0&0&0&1&1&0&0&0&1&0&1 \\ 1&1&0&0&0&0&0&0&0&1&1&0 \\ 1&1&1&1&1&0&0&0&0&0&0&1 \\
0&0&0&1&1&0&1&1&0&0&0&1 \\ 1&0&1&1&0&1&0&1&0&0&0&0 \\ 0&1&0&1&0&0&1&0&1&1&0&0
\end{array} }
\def\Mbe{ \begin{array}{c}\strut
 1\\1\\1\\1\\1\\1\\1\\1\\1\\1\\1\\1 \end{array} }

\def\Mea{ \begin{array}{c|}
\hbox to 0pt{\hss\normalsize $b$\kern0.7em}\strut 1\end{array} }
\def\Med{ \begin{array}{cccccccccccc|}\strut 0&0&0&0&0&0&0&0&0&0&0&0 \end{array} }
\def\Mec{ \begin{array}{ccccccccccc}\strut 1&1&1&1&1&1&1&1&1&1&1 \end{array} }
\def\Meb{ \begin{array}{cccccccccccc|}\strut 1&1&1&1&1&1&1&1&1&1&1&1 \end{array} }
\def\Mee{ \begin{array}{c} \strut 0\end{array} }

\newpage
\def\strut{\vrule height2.2ex width0pt}

{\small
\[ \arraycolsep=1.9pt \def\arraystretch{0.81}
 \begin{array}{c@{}c@{}c@{}c@{}c}
 \multicolumn{3}{c}{\hbox{\normalsize$G$}} &&\\[-0.4ex]
 \multicolumn{3}{c}{$\downbracefill$}&& \\[0.4ex]
  \Maa & \Mab  & \Mac & \Mad & \Mae \\
 \hline
 \Mba & \Mbb  & \Mbc & \Mbd & \Mbe \\
 \cline{3-5}
 \Mca& \Mcb  & \Mcc & \Mcd & \Mce \\
 \cline{1-3}
 \Mda & \Mdb  & \Mdc & \Mdd & \Mde \\
 \hline
 \Mea & \Meb  & \Mec & \Med & \Mee \\
 &&\multicolumn{3}{c}{$\upbracefill$} \\[0.6ex]
  &&\multicolumn{3}{c}{\hbox{\normalsize$H$}} 
\end{array} \]
}
\caption{An example of $G,H\in\cR(4,5,24)$ glued along
 $K\in\cR(3,5,11)$ (the square in the centre), to make $F\in\cR(5,5,37)$.
 \label{GlueFig}}
\end{figure}

\section{Step 3. Empirical results}
For $6\le d\le 9$, no gluings produced any output graphs, so Step~3 was
unnecessary.
For $d=10$ we found a total of 647,424 graphs (81,936 nonisomorphic)
in $\cR(5,5,38)$, all of them from a single $K \in \cR(3,5,10)$.
For $d=11$ we found a total of 15,244 graphs 
in $\cR(5,5,37)$, with $15,152$ graphs (14,412 nonisomorphic) coming from one
$K \in \cR(3,5,11)$ and 92 graphs (84 nonisomorphic) coming from another~$K$.
An example is shown in Figure~\ref{GlueFig}.
None of these graphs could be extended by one more vertex while
staying within $\cR(5,5)$, so Step~3 was completed successfully.

By Step~2, we do not need gluings for $d\ge 12$, which is fortunate
since the number of successful gluings is around 57 billion for $d=12$ and
perhaps even larger for $d=13$.  This would make Step~3 very onerous.
Of course, these considerations are the reason we sought to
eliminate $d\ge 12$ theoretically (Lemma~\ref{l:Step123ok}).

We wish to acknowledge useful comments from Staszek Radziszowski.

\end{document}